\documentclass[12pt]{amsart}
\usepackage{amsmath,amsthm,latexsym,amscd,amsbsy,amssymb,url}
\usepackage[all]{xypic}
 \relax

\chardef\bslash=`\\ 

\makeatletter
\def\verbatim{\interlinepenalty\@M \@verbatim
  \leftskip\@totalleftmargin\advance\leftskip2pc
  \frenchspacing\@vobeyspaces \@xverbatim}
\makeatother
\newtheorem{thm}{Theorem}[section]
\newtheorem{cor}[thm]{Corollary}
\newtheorem{lem}[thm]{Lemma}

\theoremstyle{definition}
\newtheorem{defin}{Definition}[section]

\theoremstyle{remark}

\numberwithin{equation}{section}

\begin{document}


\title
{There are no noncommutative soft maps}
\author{A. Chigogidze}
\address{Department of Mathematics and Statistics,
University of North Carolina at Greensboro,
317 College Avenue, 126 Petty Bldg, Greensboro, NC, 27402, USA}
\email{chigogidze@uncg.edu}


\begin{abstract}{It is shown that for a map $f \colon X \to Y$ of compact spaces the unital $\ast$-homomorphism $C(f) \colon C(Y) \to C(X)$ is projective in the category $\operatorname{Mor}({\mathcal C}^{1})$ precisely when $X$ is a dendrite and $f$ is either homeomorphism or a constant.}
\end{abstract}

\maketitle
\markboth{A.~Chigogidze}
{Noncommutative soft maps}

\section{Introduction}
By Gelfand's duality any topological property of a categorical nature in the category $\mathcal{COMP}$ (= compact spaces and their continuous maps) has its counterpart in the category $\mathcal{AC}^{1}$ (= commutative unital $C^{\ast}$-algebras and their unital $\ast$-homomorphisms) which, in turn, serves as a prototype for the corresponding concept in the larger category ${\mathcal C}^{1}$ (= unital $C^{\ast}$-algebras and their unital $\ast$-homomorphisms). 

For example, $X$ is an injective object in $\mathcal{COMP}$ (i.e. $X$ is a compact absolute retract) precisely when the $C^{\ast}$-algebra $C(X)$ is a projective object in $\mathcal{AC}^{1}$. However, requirement that $C(X)$ is actually projective object in the full category ${\mathcal C}^{1}$ imposes severe restrictions back on $X$: as shown in \cite{CD} this happens if and only if $X$ is a dendrit (i.e. at most one dimensional metrizable $AR$-compactum). In other words, the class of dendrits coincides with the class of noncommutative absolute retracts.

Expanding further to the category $\operatorname{Mor}(\mathcal{COMP})$ we note that injective objects in it are also well understood and play important role in geometric topology. These are soft maps between $AR$-compacta. Recall (see, for instance, \cite[Definition 2.1.33]{chibook}) that a map $f \colon X \to Y$ of compact spaces is soft if for any compact space $B$, any closed subset $A \subset B$, and any two maps $g \colon A \to X$ and $h \colon B \to Y$ such that $f \circ g = h|A$, there exists a map $k \colon B \to X$ such that $g = k|A$ and $f \circ k = h$. Here is the diagram illustrating the situation:

\[
        \xymatrix{
              A \ar_{\operatorname{incl}}@{_{(}->}[d] \ar^{g}[rr]& & X \ar_{f}[d] \\
              B \ar^{h}[rr] \ar@{.>}^{k}[urr]& & Y  \\
    }
\]

As noted, by reversing arrows and allowing all (not necessarily commutative) unital $C^{\ast}$-algebras, we arrive to the following concept of doubly projective homomorphism. This concept was first introduced in \cite[Definition 3.1]{LP} and studied also in \cite{C}. It must be noted that in \cite{C}, as well as below, we do not assume (while \cite{LP} does) that the domain of doubly projective homomorphism is projective.

\begin{defin}\label{D:1}
A unital $\ast$-homomorphism $i \colon X \to Y$ of unital $C^{\ast}$-algebras is doubly projective if for any unital $\ast$-homomorphisms $f \colon X \to A$, $g \colon Y \to B$ and any surjective unital $\ast$-homomorphism $p \colon A \to B$ with $g \circ i = p \circ f$, there exists a unital $\ast$-homomorphism $h \colon Y \to A$ such that $f = h \circ i$ and $g = p\circ h$. In other words, any commutative diagram (of unbroken arrows) 

\[
        \xymatrix{
              B & & Y \ar_{g}[ll] \ar@{.>}_{h}[lld]\\
              A \ar^{p}[u] & & X \ar_{f}[ll] \ar_{i}[u]\\
    }
\]

\noindent with surjective $p$ can be completed by the dotted diagonal arrow with commuting triangles.
\end{defin}

\begin{lem}\label{L:retract}
Retract of a doubly projective homomorphism is doubly projective. More precisely, suppose that $i_{1} \colon X_{1} \to Y_{1}$ is doubly projective, and for a unital $\ast$-homomorphism $i_{2} \colon X_{2} \to Y_{2}$ there exist unital homomorphisms $s_{1} \colon X_{2} \to X_{1}$, $r_{1} \colon X_{1} \to X_{2}$ with $r_{1}\circ s_{1} = \operatorname{id}_{X_{2}}$ and $s_{2} \colon Y_{2} \to Y_{1}$, $r_{2} \colon Y_{1} \to Y_{2}$ with $r_{2}\circ s_{2} = \operatorname{id}_{Y_{2}}$. If $i_{1}\circ s_{1} = s_{2}\circ i_{2}$, then $i_{2}$ is also doubly projective.
\end{lem}
\begin{proof}
Consider the following diagram of unbroken arrows 
\[
        \xymatrix{
              B & & Y_{2} \ar_{g}[ll] \ar@{.>}_{h}[lld]\\
              A \ar^{p}[u] & & X_{2} \ar_{f}[ll] \ar_{i_{2}}[u]\\
    }
\]

\noindent with $p$ is surjective as in Definition \ref{D:1}. We need to construct a dotted $\ast$-homomorphism $h$ making both triangular diagrams commutative.

Since $i_{1}$ is doubly projective there exists a unital $\ast$-homomorphism $q \colon Y_{1} \to A$ such that $g \circ r_{2} = p\circ q$ and $f\circ r_{1} = q\circ i_{1}$. Here is the full diagram

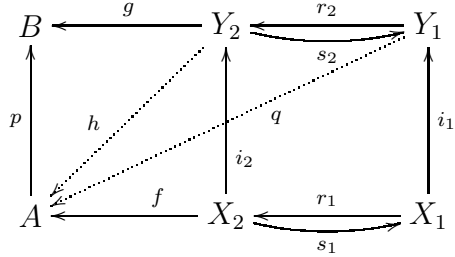
\[
        \xymatrix{
              B & & Y_{2}  \ar_{g}[ll] \ar@{.>}_(0.6){h}[lldd] \ar@/_0.5pc/_{s_{2}}[rr]& & Y_{1} \ar@{.>}^(0.40){q}[lllldd] \ar_{r_{2}}[ll]\\
 & & & & \\
              A \ar^{p}[uu] & & X_{2} \ar_(0.35){f}[ll] \ar_(0.3){i_{2}}[uu] \ar@/_0.5pc/_{s_{1}}[rr] & & X_{1} \ar_{r_{1}}[ll] \ar_{i_{1}}[uu]\\
    }
\]

\noindent Let $h = q\circ s_{2}$. It only remains to note that
\[ f = f \circ r_{1}\circ s_{1} = q \circ i_{1}\circ s_{1} = q\circ s_{2}\circ i_{2} = h\circ i_{2} \]

\noindent and

\[ g = g \circ r_{2}\circ s_{2} = p \circ q \circ s_{2} = p \circ h .\]
\end{proof}

\begin{thm}\label{T:main}
Let $f \colon X \to Y$ be a surjective map of a compact space $X$ onto a non-trivial Peano continuum $Y$. If $C(f) \colon C(Y) \to C(X)$ is doubly projective, then $f$ is a homeomorphism.
\end{thm}
\begin{proof}
Assume the contrary and let $y_{0} \in Y$ be point such that $|f^{-1}(y_{0})| > 1$. Since $C(f)$ is doubly projective in the category ${\mathcal C}^{1}$ it is doubly projective in the smaller category ${\mathcal AC}^{1}$. By Gelfand's duality the latter means precisely that $f$ is a soft map. Choose points $x_{0}, x_{1} \in f^{-1}(y_{0})$ with $x_{0} \neq x_{1}$. Softness of $f$ guarantees that there exist two sections $i_{0}, i_{1} \colon Y \to X$ of $f$ such that $i_{k}(y_{0}) = x_{k}$ for each $k = 0,1$. Note that the set $V = \{ y \in Y \colon i_{0}(y) \neq i_{1}(y)\}$ is a non-empty (since $y_{0} \in V$) open subset of $Y$ and in view of our assumption contains a homeomorphic copy of the segment $[0,1] \subset V$ (i.e. geodesic segment in $V$ between two points - denoted by $0$ and $1$). Let $Z = f^{-1}([0,1])$ and fix a retraction $r \colon Y \to [0,1]$. Since $f|Z \colon Z \to [0,1]$ is soft there exists a retraction $s \colon X \to Z$ such that $f\circ s = r \circ f$. Then, by Lemma \ref{L:retract}, $C(f|Z) \colon C([0,1]) \to C(Z)$ is doubly projective. Since $C([0,1])$ is projective in ${\mathcal C}^{1}$ we conclude by \cite[Lemma 5.3]{C} that $C(Z)$ is projective in ${\mathcal C}^{1}$. Consequently, by \cite[Theorem 4.3]{CD}, $Z$ is a dendrite, in particular, $\dim Z =1$. Consider the fiber $f^{-1}(0) \subset Z$. Since $f$ is soft, $f^{-1}(0)$ is a non-trivial absolute retract and consequently contains a segment $[i_{0}(0), i_{1}(0)]$ connecting the points $i_{0}(0)$ and $i_{1}(0)$. Similarly, fiber $f^{-1}(1)$ contains a segment $[i_{0}(1), i_{1}(1)]$ connecting the points $i_{0}(1)$ and $i_{1}(1)$. Clearly the union $S$ of these four segments $[i_{0}(0), i_{1}(0)]$, $i_{1}([0,1])$, $[i_{0}(0), i_{1}(0)]$ and $i_{0}([0,1])$ is homeomorphic to the circle $S^{1}$. Since $\dim Z = 1$, there exists retraction $p \colon Z \to S$. But this is impossible because $Z$ is an absolute retract.
\end{proof}

\begin{cor}
Let $f \colon X \to Y$ be a map of compact spaces. Then the following conditions are equivalent:
\begin{itemize}
  \item[(i)]
$C(f) \colon C(Y) \to C(X)$ is a projective object of the category $\operatorname{Mor}({\mathcal C}^{1})$; 
  \item[(ii)] 
$X$ is a dendrit and $f$ is either a homeomorphism or a constant map.   
\end{itemize}
\end{cor}
\begin{proof}
(i) $\Longrightarrow$ (ii). General nonsense easily implies that both $C(X)$ and $C(Y)$ are projective in ${\mathcal C}^{1}$. Thus, by \cite{CD}, $X$ and $Y$ are dendrits. Also \cite[Proposition 5.11]{C} guarantees that $C(f)$ is doubly projective. By \ref{T:main}, $f$ is either constant or a homeomorphism. 

(ii) $\Longrightarrow$ (i) is trivial.
\end{proof}


\end{document}